\documentclass[letterpaper, 10 pt, conference]{ieeeconf}  % Comment this line out
\usepackage[utf8]{inputenc}
\usepackage[english]{babel}
\usepackage{mathrsfs} % To use the \mathscr commands
\usepackage{amsmath, amssymb, amsfonts}

\IEEEoverridecommandlockouts                              % This command is only
\usepackage{xcolor}
\usepackage[pdftex]{graphicx}
  \graphicspath{{./imgs/}}
  \DeclareGraphicsExtensions{.pdf,.jpeg,.png}
  
\usepackage{epstopdf} % Must be loaded right after graphicx
\usepackage{epsfig}   % for postscript graphics files

\usepackage{etoolbox}

\newtoggle{arXiv}
\toggletrue{arXiv}
\iftoggle{arXiv}{

}{

\usepackage{hyperref} 

\hypersetup{
colorlinks=true, linktocpage=true, pdfstartpage=1, pdfstartview=FitV,
breaklinks=true, pdfpagemode=UseNone, pageanchor=true, pdfpagemode=UseOutlines,
plainpages=false, bookmarksnumbered, bookmarksopen=true, bookmarksopenlevel=1,
hypertexnames=true, pdfhighlight=/O, urlcolor=red!50!black, linkcolor=red!50!black, citecolor=green!50!black,
pdftitle={},
pdfauthor={},
pdfsubject={},
pdfkeywords={},
pdfcreator={},
pdfproducer={LaTeX}
}
}

\usepackage[natbib=true,backend=biber,style=ieee,
firstinits=true,doi=false,eprint=false,isbn=false,url=false,texencoding=utf8,bibencoding=utf8]{biblatex}
\AtBeginBibliography{\footnotesize}
\IfFileExists{./Library.bib}{

 \addbibresource{./Library.bib}
 
}{

  \addbibresource[location=remote]{https://www.dropbox.com/s/r00o9lw76cksxvj/Library.bib?dl=1} 

}

\usepackage{todonotes}
\usepackage{enumerate}

\usepackage{amsthm}
\theoremstyle{plain}
\newtheorem{theorem}{Theorem}
\newtheorem{proposition}{Proposition}

\newtheorem{corollary}{Corollary}

\theoremstyle{definition}
\newtheorem{definition}{Definition}

\newtheorem{problem}{Problem}

\theoremstyle{remark}

\title{\LARGE \bf Decentralized Nonlinear Feedback Design with Separable Control
Contraction Metrics}

\author{Humberto Stein Shiromoto and Ian R. Manchester% <-this % stops a space
\thanks{This work was supported by the Australian Research Council.}% <-this % stops a space
\thanks{Both authors are with The Australian Centre for Field Robotics and the School of Aerospace Mechanical and Mechatronical Engineering, The University of Sydney, 2006 NSW, Australia. Corresponding author: 
        {humberto.shiromoto@ieee.org, orcid.org/0000-0002-0883-0231}}%
}

\begin{document}

\maketitle
\thispagestyle{empty}
\pagestyle{empty}

\begin{abstract}
 The problem under consideration is the synthesis of a distributed controller for a nonlinear network composed of input affine systems. The objective is to achieve exponential convergence of the solutions. To design such a feedback law, methods based on contraction theory are employed to render the controller-synthesis problem scalable and suitable to use distributed optimization. The nature of the proposed approach is constructive, because the computation of the desired feedback law is obtained by solving  a convex optimization problem. An example illustrates the proposed methodology.
\end{abstract}

%\begin{IEEEkeywords} Stability of nonlinear systems, Distributed control, Large-scale %systems
%\end{IEEEkeywords}

\section{Introduction}\label{sec:introduction}

In many applications, nonlinear systems appear in form of interconnections of simpler elements, for instance, when employing techniques to model complex systems \cite{Willems2007}, and where coordination is important (see \cite{Nijmeijer2003,Pettersen2006} and references therein).

Techniques to design control algorithms for nonlinear systems depend on the structure of the differential equations. For nonlinear input-affine systems described by continuous equations typical approaches include the design of a control-Lyapunov function (CLF) \cite{Sontag1998}. From the constructive viewpoint, this approach may require to solve a non-convex optimization problem \cite{Rantzer:2001}. An alternative to CLFs is provided by the so-called control-contraction metrics (CCM) which provides a controller by solving an optimization problem that is convex \cite{Manchester2014a}. Contraction (a concept related to incremental stability) theory \cite{Angeli2002,Lohmiller1998} can be traced back to the work \cite{Lewis1949}. An advantage of contraction-based methods is that it allows to decouple the design of control algorithms from the \emph{a priori} knowledge of location of the attractor of the closed-loop system \cite{Aminzare2014}. This is particularly relevant in the case of attractiveness of an  invariant manifold which is a typical problem in network systems \cite{Wang2004,PhamSlotine2007,Andrieu2016}.

In this paper a network composed of nonlinear systems is considered. For this case, the controller designer must take into account also the influence of the network topology (see \cite{Aminzare2014a,Russo2013,Wang2004} for the contraction analysis of nonlinear network systems). Usual approaches involve the distributed design of the control algorithms employing input-output methods \cite{Scardovi2010} such as dissipativity \cite{Forni2013a}, and passivity \cite{Persis2014} which require an input-output property to hold on all agents of the network. From this point of view, scalable tools \cite{Rantzer2015,Tanaka2011} can be less restrictive allowing the decomposition (resp. composition) of the network into smaller (resp. bigger) components \cite{Dirr2015}. However, this approach is still not available for nonlinear systems.

The contribution of this paper is the use of contraction theory to extend these scalable methodologies for networks composed of nonlinear systems that are affine in the input. More specifically, the innovation of this work is twofold. An appropriate storage function (see the precise definition of this concept below) and a feedback law are designed for each node by solving a distributed convex-optimization problem. The result presented in this work is constructive, i.e., an algorithm to design the controller is provided. Moreover, the obtained controller will depend only on the states of the system to which it is designed for and on the states of adjacent systems. Although the feedback law is real-time optimized in a similar way to nonlinear model predictive control (\citep{Findeisen2003}), it has a simpler structure because it does not depend on dynamical constraints.

\textit{Outline.} The remainder of this paper is organized as follows. In Section~\ref{sec:Problem Formulation}, the problem under consideration is formulated and the preliminaries needed for the results are recalled. Section~\ref{sec:Sum-Separable Control-Contraction Metrics} presents the result that solves the problem in consideration. An example of the proposed approach is provided in Section~\ref{sec:illustration}. Section~\ref{sec:conclusion} collects final remarks and future directions of this work.

\section{Problem Formulation}\label{sec:Problem Formulation}

\textit{Notation.} Let $N\in\mathbb{N}$ be a constant value. The notation $\mathbb{N}_{[1,N]}$ stands for the set $\{i\in\mathbb{N}:1\leq i\leq N\}$. Let $c\in\mathbb{R}$ be a constant value. The notation $\mathbb{R}_{[1,c]}$ (resp. $\mathbb{R}_{\diamond c}$) stands for the set $\{x\in\mathbb{R}:1\leq x\leq c\}$ (resp.  $\{i\in\mathbb{R}:i\diamond c\}$, where $\diamond$ is a comparison operator, i.e., $\diamond\in\{<,\geq,=,\ \text{etc}\}$). A matrix $A\in\mathbb{R}^{n\times n}$ with zero elements except (possibly) those $a_{ii},\ldots,a_{nn}$ on the diagonal is denoted as $\mathbin{\mathtt{diag}}(a_{ii},\ldots,a_{nn})$. The notation $M\succ 0$ (resp. $M\succeq 0$) stands for $M$ being positive (resp. semi)definite.

A continuous function $f:\mathbf{S}\to\mathbb{R}$ defined in a subset $\mathbf{S}$ of $\mathbb{R}^k$ containing the origin is \emph{positive definite} if, for every $x\in\mathbf{S}_{\neq0}$, $f(x)>0$ and $f(0)=0$. The class of \emph{positive definite functions} is denoted as $\mathcal{P}$. It is \emph{proper} if it is radially unbounded. By $\mathcal{C}^s$ the class of \emph{$s$-times continuously differentiable functions} is denoted. In particular, the function $f$ is said to be \emph{smooth} if $s=\infty$. The notation $\mathcal{L}_{\mathrm{loc}}^\infty(\mathbb{R}_{\geq0},\mathbb{R}^m)$ stands for the class of functions $u:\mathbb{R}\to\mathbb{R}^m$ that are locally essentially bounded. Given differentiable functions $M:\mathbb{R}^n\to\mathbb{R}^{n\times n}$ and $f:\mathbb{R}^n\to\mathbb{R}^n$ the notation $\partial_fM$ stands for matrix with dimension $n\times n$ and with $(i,j)$ element given by $\frac{\partial m_{ij}}{\partial x}(x)f(x)$. The notation $f'$ stands for the total derivative of $f$.

Let $N>0$ be an integer, a \emph{graph} consists of a set of \emph{nodes} $\mathscr{V}\subset\mathbb{N}_{[1,N]}$ and a set of \emph{edges} $\mathscr{E}\subset\mathscr{V}\times\mathscr{V}$ and it is denoted by the pair $(\mathscr{V},\mathscr{E})=\mathscr{G}$. A node $i\in\mathscr{V}$ is said to be \emph{adjacent} to a node $j\in\mathscr{V}$ if $(i,j)\in\mathscr{E}$, the set of nodes that are adjacent to $j$ is defined as $\mathscr{N}(j)=\{i\in\mathscr{V}:i\neq j,(i,j)\in\mathscr{E}\}$. Given two nodes $i,j\in\mathscr{V}$, an ordered sequence of edges $\left\{(k,k+1)\right\}_{k=i}^{j-1}$ is said to be a \emph{path from the node $i$ to the node $j$}. A graph is said to be \emph{strongly connected} if, for every two nodes $i,j\in\mathscr{V}$, there exists a path connecting them. It is also said to be \emph{complete} is every node is adjacent to any other node, and it is \emph{incomplete} otherwise.

Let $N>1$ be a constant integer. Consider the network described by a strongly connected and incomplete graph with $N$ nodes. The time-evolution of the state of each node $i\in\mathbb{N}_{[1,N]}$ is given by the equation
	\begin{equation}\label{eq:subsystem}
		\dot{x}_i(t)=f_i(x_i(t),\breve{x}_i(t))+B_i(x_i(t))u_i(t),
	\end{equation}
and for every positive value of the \emph{time} $t$, the \emph{system state} $x_i(t)$ (resp. \emph{system input} $u_i(t)$) evolves in the Euclidean space $\mathbb{R}^{n_i}$ (resp. $\mathbb{R}^{m_i}$). Consider the systems indexed by $j_1,\ldots,j_k\in\mathscr{N}(i)$ that are adjacent to $i$, the vector $\breve{x}_i(t):=(x_{j_1}(t),\ldots,x_{j_k}(t))\in\mathbb{R}^{\breve{n}_i}$ is the \emph{interconnecting input}, where  $\breve{n}_i:=n_{j_1}+\cdots+n_{j_k}$. The \emph{system dynamics} is given by the smooth vector fields $f_i$ and $B_i$. From now on, the dependence of the system~\eqref{eq:subsystem} on the time variable $t$ will be omitted.

Employing a more compact notation, the network is given by the equation
\begin{equation}\label{eq:network}
	\dot{x}=f(x)+B(x)u,
\end{equation}
where $x:=(x_1,\ldots,x_N)^\top\in\mathbb{R}^n$ with $n=n_1+\cdots+n_N$, $u=(u_1,\ldots,u_N)^\top\in\mathbb{R}^m$  with $m=m_1+\cdots+m_N$, $f=(f_1,\ldots,f_N)^\top$ and $B=\mathbin{\mathtt{diag}}(B_1,\ldots,B_N)^\top$.

A function $u\in\mathcal{L}_{\mathrm{loc}}^\infty(\mathbb{R}_{\geq0},\mathbb{R}^m)$ is said to be an \emph{input signal or control for \eqref{eq:network}}. The solution $X$ to \eqref{eq:network} under the input signal $u^\ast$ with initial condition $x^\ast$ and computed at time $t$ is denoted as $X(t,x^\ast,u^\ast)$. Here, the $\ast$ stands for prescribed initial condition and input. From now on, solutions to \eqref{eq:network} are assumed to exist, for every positive time $t$.

The objective of this work is to design a distributed feedback law $k:\mathbb{R}^n\to\mathbb{R}^m$ for system \eqref{eq:network} such that, for every initial condition $x\neq x^\ast$, solutions $X(\cdot,x,k)$ to the closed-loop system converge exponentially to $X(\cdot,x^\ast,u^\ast)$.

A suitable framework to deal with the exponential convergence of pair of solutions is provided by the concept of exponential stabilizability of solutions introduced in \cite{Manchester2014a}, and recalled as follows.

\begin{definition}\label{def:GES}
	System \eqref{eq:network} is said to be \emph{universally exponentially stabilizable with rate $\lambda>0$} if, for every solution $X(\cdot,x^\ast,u^\ast)$ to \eqref{eq:network}, there exist a constant value $C>0$ and a feedback law $k$ such that, for every initial condition $x\in\mathbb{R}^n$, the corresponding solution $X(\cdot,x,k)$ to \eqref{eq:network} satisfies the inequality
	\begin{equation}\label{eq:contraction inequality}
		|X(t,x^\ast,u^\ast)-X(t,x,k)|\leq C e^{-\lambda t}|x^\ast-x|,
	\end{equation}
	for every $t\geq0$.
\end{definition}

Now that the concept of a exponential stabilizability has been recalled in Definition~\ref{def:GES}, the problem under consideration can be formulated as follows.

\begin{problem}\label{pb:formulation} For every solution $X(\cdot,x^\ast,u^\ast)$ to system \eqref{eq:network}. Find a feedback law $k:\mathbb{R}^n\to\mathbb{R}^m$ for \eqref{eq:network} such that
	\begin{enumerate}[{P}$_1$.]
		\item For every initial condition $x\in\mathbb{R}^n$, the respective issuing solution $X$ to the closed-loop system
	\begin{equation}\label{eq:network:CL}
		\dot{x}=f(x)+B(x)k(x)
	\end{equation}
	satisfies the inequality described in Equation~\eqref{eq:contraction inequality};
	
	\item For a given index $i\in\mathbb{N}_{[1,N]}$, the $i$-th component of the vector field $k$ depends only on $x_i$ and on the neighbors of the system $i$.
	\end{enumerate}
\end{problem}

\paragraph{Preliminaries} To solve Problem~\ref{pb:formulation}, it is sufficient to analyse the behavior of the solutions to the \emph{differential system associated with \eqref{eq:network}}, i.e., the time-variation of the solutions to system~\eqref{eq:network} issuing from piecewise smooth curves of $\mathbb{R}^{n\times m}$. This system is given by
\begin{equation}\label{eq:delta network}
	\dot{\delta}_x=A(x,u)\delta_ x + B(x)\delta_ u,
\end{equation}
where $\delta_x(\cdot)=(\delta_{x,1}(\cdot),\ldots,\delta_{x,N}(\cdot))$ (resp. $\delta_u(\cdot)$) is the tangent vector to the piecewise smooth curve connecting any pair of points $\underline{x},\overline{x}\in\mathbb{R}^n$ (resp. $\underline{u},\overline{u}\in\mathbb{R}^m$) \cite{Lohmiller1998}. The matrix $A(\cdot)\in\mathbb{R}^{n\times n}$ has components given as $A_{ij}=(\partial [f_i+b_iu_i])/(\partial x_j)$, for every $(i,j)\in\mathbb{N}_{[1,N]}\times \mathbb{N}_{[1,n]}$. The solution $\Delta_x$ to system ~\eqref{eq:delta network} under the input signal $\delta_u$  with initial condition $\delta_x\in\mathbb{R}^n$ and computed at time $t$ is denoted as $\Delta_x(t,\delta_x,\delta_u)$.

System \eqref{eq:delta network} describes the behaviour of solutions to \eqref{eq:network} in the space tangent to the point $(x,u)\in\mathbb{R}^n\times\mathbb{R}^m$. Asymptotic notions of solutions to system~\eqref{eq:delta network} are defined in a similar fashion as for linear time-varying systems (LTIS) (see \cite{Hespanha:2009} for more information on LTIS). Namely, given a suitable \emph{differential feedback law} $\delta_k:\mathbb{R}^n\to\mathbb{R}^m$ for system \eqref{eq:delta network}, the origin is a globally exponentially stable equilibrium for the closed-loop system 
\begin{equation}\label{eq:delta network:CL}
		\dot{\delta}_x=A(x,u)\delta_x-B(x)\delta_k(x,\delta_x),
	\end{equation}
	if, for every initial condition $\delta_x\in\mathbb{R}^n$, the limit $|\Delta_x(t,\delta_x,\delta_k)|\to0$ exponentially, as $t\to\infty$. If the origin is an exponentially stable equilibrium for system \eqref{eq:delta network:CL}, then the energy (see \eqref{eq:dist} below) of each piecewise smooth curve connecting two points in $\mathbb{R}^n$ remains constant and eventually tends to zero. 

As in the classical Lyapunov case, a sufficient condition for an equilibrium point to be globally exponentially stable for system~\eqref{eq:delta network} is provided by the analysis of an appropriate function along the solutions to systems~\eqref{eq:network} and \eqref{eq:delta network} (see \cite{Forni2014} and references therein).

\begin{definition}
	A smooth function $V:\mathbb{R}^n\times\mathbb{R}^n\to\mathbb{R}_{\geq0}$ is said to be a \emph{differential storage function for system \eqref{eq:network}} if there exist constant values $\underline{c}>0$ and $\overline{c}>0$ such that the inequality
	\begin{equation}\label{eq:Finsler-Lyapunov:candidate}
		\underline{c}|\delta_x|^2\leq V(x,\delta_x)\leq \overline{c}|\delta_x|^2
	\end{equation}
	holds, for every $(x,\delta_x)\in\mathbb{R}^n\times\mathbb{R}^n$. Given fixed controls $u$ and $\delta_u$ for, respectively, systems~\eqref{eq:network} and \eqref{eq:delta network}, the differential storage function $V$ is said to be a \emph{contraction metric for system \eqref{eq:network}} if there exists a constant value $\lambda>0$ such that the inequality
	\begin{align}
		\tfrac{dV}{dt}(X(t,x,u),\Delta_x(t,\delta_x,\delta_u))\leq  -2\lambda V&(x,\delta_x)\label{eq:Finsler-Lyapunov:inequality}
	\end{align}
	holds, for every $(x,\delta_x)\in\mathbb{R}^n\times\mathbb{R}^n$.
\end{definition}

In other words, Equation~\eqref{eq:Finsler-Lyapunov:candidate} states that the function $V$ is positive definite and proper with respect to $\delta_x$ in the space tangent to $x$. Equation~\eqref{eq:Finsler-Lyapunov:inequality} states that $V$ is strictly decreasing along the solutions of the system composed of \eqref{eq:network} and \eqref{eq:delta network}.

The existence of a contraction metric for system \eqref{eq:network} is a sufficient condition to conclude that \eqref{eq:network} is contracting (see \cite{Andrieu2016,Forni2014,Lohmiller1998} and references therein).

A contraction metric can be defined in terms of Riemannian metrics \cite[Definition 2.6]{Boothby1986} as follows. Let the constant values $\underline{m}>0$ and $\overline{m}>0$, and consider a smooth function $M:\mathbb{R}^n\to\mathbb{R}^{n\times n}$ with $M(\cdot)$ symmetric satisfying the inequality
\begin{equation}\label{eq:metric matrix}
	\underline{m}I\leq M(x)\leq \overline{m}I,
\end{equation}
for every $x\in\mathbb{R}^n$. The need for the bounds \eqref{eq:metric matrix} is related to the existence of suitable curves connecting any two points of $\mathbb{R}^n$. As in \cite{Manchester2014} a \emph{metric for system \eqref{eq:network}} is defined, for every $(x,\delta_x)\in\mathbb{R}^n\times\mathbb{R}^n$, by
\begin{equation}\label{eq:FL function}
	V(x,\delta_x)=\delta_x^\top M(x)\delta_x.
\end{equation}

Due to \eqref{eq:metric matrix}, the function $V$ satisfies the inequality \eqref{eq:Finsler-Lyapunov:candidate}. Moreover, whenever its time derivative computed along solutions to systems~\eqref{eq:network} and \eqref{eq:delta network} also satisfies the inequality \eqref{eq:Finsler-Lyapunov:inequality}, system \eqref{eq:network} is contractive. For a proof, the reader may check \cite[Theorems 5.7 and 5.33]{Isac2008}. 

A differentiable feedback law for system \eqref{eq:delta network} and a feedback law for system \eqref{eq:network} are obtained using the following notion of contraction metric, recalled from \cite{DBLP:journals/corr/ManchesterS15}.

\begin{definition}\label{def:existence of a metric}
	A \emph{strong control-contraction metric for system \eqref{eq:network}} is a contraction metric satisfying inequality \eqref{eq:metric matrix} and such that the condition
	\begin{subequations}\label{eq:Artstein-Sontag}
	\begin{equation}
		\delta_x\neq0\quad \text{ and }\quad\delta_x^\top M(x)B(x)=0
	\end{equation}
	implies that the inequality
	\begin{align}
		\delta_x^\top\left(\partial_f M(x)+\tfrac{\partial f}{\partial x}^\top (x)M(x)+M(x)\tfrac{\partial f}{\partial x}(x)\right)\delta_x\nonumber\\
		<-2\lambda \delta_x^\top M(x)&\delta_x\label{eq:Artstein-Sontag:inequality}
	\end{align}
	\end{subequations}
	holds. Furthermore, it also satisfies the identity
	\begin{equation}\label{eq:invariant vector}
		\partial_BM(x)+\tfrac{\partial B}{\partial x}^\top(x) M(x)+M(x)\tfrac{\partial B}{\partial x}(x)\equiv0.
	\end{equation}
\end{definition}

The following result, recalled from \cite[Theorem 1]{Manchester2014a} (see \cite{DBLP:journals/corr/ManchesterS15} for a proof), formalizes how a feedback law is obtained from a control-contraction metric for system \eqref{eq:network}.

\begin{proposition}\label{prop:existence of a differential controller}
	If there exists a strong control-contraction metric for system \eqref{eq:network}, then there exists feedback law $k$ for system \eqref{eq:network} that solves item P$_1$ of Problem~\ref{pb:formulation}. 
\end{proposition}

Note that, because of the lack of structure of $M$, the feedback law provided by Proposition~\ref{prop:existence of a differential controller} does not have any particular structure, \emph{a priori}. Moreover, as remarked in \cite{Rantzer2015} and \cite{Tanaka2011}, the computation of a solution to the set of equations \eqref{eq:Artstein-Sontag} and \eqref{eq:invariant vector} may not be done in a distributed fashion. In comparison to existing methods to achieve exponential convergence of solution, the main advantage of the proposed approach is its formulation as a convex optimization problem. This motivates the approach of proposed in this paper.

\section{Sum-Separable Control-Contraction Metrics}\label{sec:Sum-Separable Control-Contraction Metrics}

\begin{definition}
	A strong control-contraction metric $V:\mathbb{R}^n\times\mathbb{R}^n\to\mathbb{R}_{\geq0}$ for system \eqref{eq:network} receives the adjective \emph{sum-separable} if, for every index $i\in\mathbb{N}_{[1,N]}$, there exist constant values $\underline{m}_i>0$ and $\overline{m}_i>0$, and a smooth function $M_i:\mathbb{R}^{n_i}\to\mathbb{R}^{n_i\times  n_i}$ with $M_i(\cdot)=M_i(\cdot)^\top$ satisfying the inequality $\underline{m}_i I_{n_i}\leq M_i(\cdot)\leq \overline{m}_i I_{n_i}$, where $I_{n_i}\in\mathbb{R}^{n_i\times n_i}$ is the identity matrix. Also, for every $(x,\delta_x)\in\mathbb{R}^n\times\mathbb{R}^n$, $V(x,\delta_x)=\delta_x^\top M(x)\delta_x=\sum_{i=1}^N\delta_{x,i}^\top M_i(x_i)\delta_{x,i}$.
\end{definition}

Although the requirement for $M$ to have a diagonal structure can be restrictive, for positive linear time-invariant systems the existence of a positive definite matrices $P=P^\top$ with diagonal structure used as Lyapunov functions is not conservative \cite{Rantzer2015}. This is an open problem that could be addressed for monotone systems. The main result of this paper is stated below. 

\begin{theorem}\label{prop:distributed controller}
	Assume that, for each index $i\in\mathbb{N}_{[1,N]}$, there exist smooth functions $W_i:\mathbb{R}^{n_i}\to\mathbb{R}^{n_i\times n_i}$ and $\rho_i:\mathbb{R}^{n_i+\breve{n}_i}\to\mathbb{R}$ such that, for the function $T:\mathbb{R}^n\to\mathbb{R}^{n\times n}$ with elements $i,j\in\mathbb{N}_{[1,N]}$ defined by
	\begin{align*}
		T_{ii}=&-\partial_{f_i}W_i+\tfrac{\partial f_i}{\partial x_i}W_i+W_i\tfrac{\partial f_i}{\partial x_i}^\top -\rho_i(\widetilde{x}_i)B_i B_i^\top+2\lambda W_i,
	\end{align*}
	where $\widetilde{x}_i:=(x_i,\breve{x}_i)\in\mathbb{R}^{n_i}\times\mathbb{R}^{\breve{n}_i}$, and
	\begin{equation*}
		T_{ij}=\left\{\begin{array}{rcl}
		  W_i\frac{\partial f_i}{\partial x_j}+\frac{\partial f_j}{\partial x_i}^\top W_j,&\text{if}&j\in\mathscr{N}(i)\\
		  0,&\multicolumn{2}{l}{\text{otherwise,}}	
		\end{array}\right.
	\end{equation*}
	the matrix inequality $T(x)\prec0$ holds, for every $x\in\mathbb{R}^n$, and the identity 
	\begin{equation}\label{eq:invariant vector:convex}
\partial_BW(x)-\tfrac{\partial B(x)}{\partial x}W(x)-W(x)\tfrac{\partial B(x)}{\partial x}^\top\equiv0
\end{equation}
	 holds. Then, there exists a sum-separable control-contraction metric and a feedback law for system \eqref{eq:network} that solves Problem~\ref{pb:formulation}.
\end{theorem} 

\begin{proof}
	The proof of Theorem~\ref{prop:distributed controller} consists of two steps and it is based on \cite[Theorem 1]{Manchester2014a}. The first step is an \enquote{offline} search of a strong control-contraction metric for system~\eqref{eq:network} in terms of a convex optimization program. The second step is an \enquote{online} the integration of the obtained controller along appropriate curves connecting points of $\mathbb{R}^n$.
	
	{\itshape First step.} The offline computation of a control-contraction metric and a differentiable feedback law for system \eqref{eq:network} can be formulated in terms of a convex optimization programming problem as follows.
	
	Under the hypothesis of Theorem~\ref{prop:distributed controller}, two equations are satisfied. Namely, the inequality 
	\begin{align}
		\eta_x^\top\left(-\partial_f W(x)+\tfrac{\partial f}{\partial x}(x)W(x)+W(x)\tfrac{\partial f}{\partial x}^\top (x)\right)\eta_x\nonumber\\
		<-2\lambda \eta_x^\top W(x)\eta_x\label{eq:Artstein-Sontag:convex inequality}
	\end{align}	
which holds, whenever $\eta_x\neq0$ and $\eta_x^\top B(x)=0$, where $\eta_x\in\mathbb{R}^n$, and the identity \eqref{eq:invariant vector:convex}.
	
	For every $(x,\delta_x)\in\mathbb{R}^n\times\mathbb{R}^n$, let $\eta_x=M(x)\delta_x$ and define the function $W(x)=M^{-1}(x)$. This change of variable implies that the set of equations \eqref{eq:Artstein-Sontag} and the identity \eqref{eq:Artstein-Sontag:convex inequality} hold. Thus, $W^{-1}$ is a strong control-contraction metric for system \eqref{eq:network}. Consequently, the function 
\begin{equation}\label{eq:differential feedback law}
		\delta_k(x,\delta_x)=-\tfrac{\rho(x)}{2}B(x)^\top W(x)^{-1}\delta_x:=K(x)\delta_x
\end{equation}
defined, for every $(x,\delta_x)\in\mathbb{R}^n\times\mathbb{R}^n$, is a differential feedback law for system \eqref{eq:delta network} that renders the origin globally exponentially stable for the closed-loop system \eqref{eq:delta network:CL}.

Once the differential feedback law \eqref{eq:differential feedback law} has been obtained, the second step is the computation of a feeback law for system \eqref{eq:network}.

{\itshape Second Step.} The online integration of \eqref{eq:differential feedback law} along appropriate curves is defined as follows.

Consider two points $x^\ast,x\in\mathbb{R}^n$. Let $\Gamma(x^\ast,x)$ be the set of piecewise smooth curves $c:[0,1]\to\mathbb{R}^n$ connecting $x^\ast$ to $x$. A curve satisfying the optimization problem
\begin{align}\label{eq:dist}
	\inf_{c\in\Gamma(x^\ast,x)}e(c):=\inf_{c\in\Gamma(x^\ast,x)}\int_0^1 V(c(s),c'(s))\,ds\nonumber\\
	=\inf_{c\in\Gamma(x^\ast,x)}\int_0^1\sum_{i=1}^N c_i'(s)^\top M_i(c_i(s))c_i'(s)\,ds
\end{align}
is said to be a \emph{geodesic curve} while the function $e$ is said to be the \emph{energy} of the curve $c$.

Since $M$ is positive definite, the minimum of this equation corresponds to the minimum of each component $i\in\mathbb{N}_{[1,N]}$, i.e., a solution to the optimization problem
\begin{equation}\label{eq:dist:i}
	\inf_{c_i\in\Gamma(x_i^\ast,x_i)}e(c_i):=\mathbin{\mathtt{dist}}(x_i^\ast,x_i).
\end{equation}

	Due to the fact that, for each index $i\in\mathbb{N}_{[1,N]}$, the metric space $(\mathbb{R}^{n_i},\mathbin{\mathtt{dist}})$ is complete and connected, from the Hopf-Rinow Theorem \cite[Theorem 7.7]{Boothby1986}, for each $x_i,x_i^\ast\in\mathbb{R}^{n_i}$ there exists a solution to the optimization problem \eqref{eq:dist:i}. Thus, the geodesic connecting $x$ to $x^\ast$ is given by $c=(c_1,\ldots,c_N)$.

At each time $t\geq0$, and for every index $i\in\mathbb{N}_{[1,N]}$, compute the geodesic $c_{i,t}:[0,1]\to\mathbb{R}^n$ connecting the $i$-th component of the solutions $X(t^-,x^\ast,u^\ast)$ and $X(t^-,x,u)$, where $t^-:=\inf\{\mathbb{R}_{\geq0}\setminus[0,t)\}$. At every time $t\geq0$, each component $i\in\mathbb{N}_{[1,N]}$ of the variational feedback law \eqref{eq:differential feedback law} integrated along $c$ yields the function
\begin{align}\label{eq:feedback law}
		k_i(t)=u_i^\ast(t)-\int_0^1\tfrac{1}{2}\rho_i(\tilde{c}_{i,t}(s))B(c_{i,t}(s))^\top\nonumber\\  W(c_{i,t}(s))^{-1}c_{i,t}'&(s)\,ds\;,
\end{align}
where $\tilde{c}_{i,t}=(c_{i,t},\breve{c}_{i,t})$. Apply the feedback law $k=(k_1,\ldots,k_N)$ to system~\eqref{eq:network} and repeat the second step.

Since the origin is a globally exponentially stable equilibrium for system~\eqref{eq:delta network:CL}, Coppel's inequality (\cite[Theorem 1]{Hewer1974} see also \cite{Sontag2010}) implies that there exist constant values $C>0$ and $\lambda>0$ such that, for every $x\in\mathbb{R}^n$ the issuing solutions $X(\cdot,x,k)$ and $X(\cdot,x^\ast,u^\ast)$ to system \eqref{eq:network:CL} satisfy, for every $t\geq0$, the inequality \eqref{eq:contraction inequality}. This concludes the proof of Theorem~\ref{prop:distributed controller}.
\end{proof}

Note that, under the assumptions of Theorem~\ref{prop:distributed controller}, the matrices $W=\mathbin{\mathtt{diag}}(W_1,\ldots,W_N)$ and $R=\mathbin{\mathtt{diag}}(\rho_1I_{m_1},\ldots,\rho_NI_{m_N})$, where $I_{m_i}\in\mathbb{R}^{m_i\times m_i}$ is the identity matrix, satisfy the inequality
	\begin{align}
	 T(x):=-\partial_f M(x)+\tfrac{\partial f}{\partial x}(x)W(x)+W(x)\tfrac{\partial f}{\partial x}^\top (x)\nonumber\\
		-B(x)R(x)B(x)^\top+2\lambda W(x)\prec0,\label{eq:convex problem:separable rho}
	\end{align}
	for every $x\in\mathbb{R}^n$.
	
	To solve the matrix inequality \eqref{eq:convex problem:separable rho}, a large-scale  optimization algorithm dedicated to deal with spasity (see, for instance, \cite{Waki2006}) can be employed, as each component $i\in\mathbb{N}_{[1,N]}$ depends only on its neighbours. Also, scalable numerical methods can be used to deal with sparsity for sums of squares and semidefinite programming \cite{Ahmadi2014}.
	
	A consequence of Theorem~\ref{prop:distributed controller} is that, on compact sets, the matrix $R$ satisfying the matrix inequality \eqref{eq:convex problem:separable rho} is constant.

\begin{corollary}\label{prop:convex formulation}
	Assume that the hypotheses of Theorem~\ref{prop:distributed controller} hold for every compact set $\mathbf{S}\subset\mathbb{R}^n$. Then, for every solution $X_\mathbf{S}(\cdot,x^\ast,u^\ast)\in\mathbf{S}$ to system \eqref{eq:network}, there exist constants $C_\mathbf{S}$, $\lambda_\mathbf{S}$, and a feedback law $k_\mathbf{S}:\mathbb{R}^n\to\mathbb{R}^m$ such that 
	\begin{enumerate}
		\item For every initial condition $x\in\mathbf{S}$, inequality \eqref{eq:contraction inequality} holds;
		
		\item The $i$-th component, where $i\in\mathbb{N}_{[1,N]}$, of the feedback law $ k_\mathbf{S}$ depends only on the variable $x_i\in\mathbb{R}^{n_i}$.
	\end{enumerate}		
\end{corollary}

\begin{proof}(of Corollary~\ref{prop:convex formulation})
Let $\bar{\rho}_i=\sup\{\rho_i(x_i,\breve{x}_i):x\in\mathbf{S}\}$ which exists because $\mathbf{S}$ is compact. Since inequality \eqref{eq:convex problem:separable rho} holds on every compact set $\mathbf{S}$, this implies that it is also satisfied with the matrix $\overline{R}=\max_{i=1,\ldots,N}\{\overline{\rho}_i\}\mathbin{\mathtt{diag}}(I_{m_1},\ldots,I_{m_N})$. Consequently the set of inequalities \eqref{eq:Artstein-Sontag} and the identity \eqref{eq:invariant vector} hold. Thus, the origin is a locally asymptotically stable equilibrium for system \eqref{eq:delta network:CL}.

The remainder of the proof is parallel to the proof of Theorem~\ref{prop:distributed controller}.
	\end{proof}
\section{Illustration}\label{sec:illustration}

Consider the network described by the undirected graph shown in Figure~\ref{fig:graph}.

\begin{figure}[htbp!]
	\centering
	\resizebox{0.6\linewidth}{!}{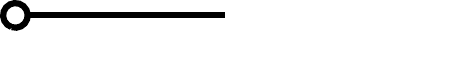}
	\caption{Graph describing the network structure of the example considered.}
	\label{fig:graph}
\end{figure}

The dynamics of each subsystem $i\in\mathbb{N}_{[1,3]}$ is given by
\begin{equation}\label{eq:example:subsytem}
	\scalebox{0.9}{$\left\{\begin{array}{rcl}
		\dot{x}_i&=&\ -x_i+z_i-1\times10^{-3}\left(x_i - \sum_{j\in\mathscr{N}(i)}x_j\right)\\
	\dot{y}_i&=&\ x_i^2 - y_i^3 -2x_iz_i+z_i\\
	\dot{z}_i&=&\ -y_i+u_i
	\end{array}\right.$}
\end{equation}
which corresponds to the system described in \cite{Andrieu2010} and interconnected using the variable $x$.

Let $q_i:=(x_i,y_i,z_i)$. Employing the optimization toolbox Yalmip \cite{Loefberg2004,Loefberg2009} for Matlab and using Mosek to solve the matrix inequality \eqref{eq:convex problem:separable rho}. The components of the matrix $W(q)=\mathbin{\mathtt{diag}}(W_1(q_1),W_2(q_2),W_3(q_3))$ are shown in the set of equations \eqref{eq:example:components of W}.

%\begin{figure*}
\begin{subequations}\label{eq:example:components of W}
\begin{equation}
		W_1(q_1)=\scalebox{0.8}{$\begin{bmatrix}
 0.06 & -0.11 x_1-0.01 & -0.16 \\
 * & 0.22 x_1^2+0.05 x_1+2.61 & 0.32 x_1+0.28 \\
 * & * & 0.89 \\
		\end{bmatrix},$}
\end{equation}

\begin{equation}
	W_2(q_2)=\scalebox{0.8}{$\begin{bmatrix}
 0.04 & -0.07 x_2-0.01 & -0.15 \\
 * & 0.15 x_2^2+0.05x_2+3.17 & 0.29x_2+0.29 \\
 * & * & 0.87 \\
	\end{bmatrix},$}
\end{equation}

\begin{equation}
	W_3(q_3)=\scalebox{0.8}{$\begin{bmatrix}
 0.09 & -0.19 x_3-0.01 & -0.23 \\
 *    & 0.38 x_3^2+0.04x_3+2.86 & 0.46 x_3+0.34 \\
 *    & * & 1.09 \\
	\end{bmatrix}.$}
\end{equation}
\end{subequations}
%\end{figure*}

To satisfy the identity \eqref{eq:invariant vector:convex}, the matrices $W_i$ do not depend on the element $z_i$. 

The components of the matrix $R(q)=\mathbin{\mathtt{diag}}(\rho_1(q_1,q_2),\rho_2(q),\rho_3(q_2,q_3))$ are shown in the set of equations \eqref{eq:example:components of R}.

\begin{subequations}\label{eq:example:components of R}
\begin{align}
	\rho_1(q_1,q_2)=3.04 x_1^2-0.00290 x_1 x_2-0.0837 x_1y_1\nonumber\\
	-0.0406 x_1 z_1-0.252 x_1+2.75 x_2^2+2.77 y_1^2\nonumber\\
	+0.0301 y_1 z_1+0.0725 y_1+2.75 y_2^2+2.75 z_1^2\nonumber\\
	+0.112 z_1+2.75 z_2^2+5.39,
\end{align}

\begin{align}
	\rho_2(q)=3.78 x_1^2-0.00220 x_1 x_2+4.08 x_2^2\nonumber\\
	-0.00220 x_2 x_3-0.0569 x_2 y_2-0.0258 x_2 z_2\nonumber\\
	-0.321 x_2+3.78 x_3^2+3.78 y_1^2+3.79 y_2^2\nonumber\\
	+0.0162 y_2 z_2+0.0666 y_2+3.78 y_3^2+3.78 z_1^2\nonumber\\
	+3.78 z_2^2+0.0999 z_2+3.78 z_3^2+6.84,
\end{align}

\begin{align}
	\rho_3(q_2,q_3)=2.75 x_2^2-0.00290 x_2 x_3+3.04 x_3^2\nonumber\\
	-0.0837 x_3 y_3-0.0406 x_3 z_3-0.252 x_3+2.75 y_2^2\nonumber\\
	+2.77 y_3^2+0.0301 y_3 z_3+0.0725 y_3+2.75 z_2^2\nonumber\\
	+2.75 z_3^2+0.112 z_3+5.39.
\end{align}
\end{subequations}

Note that, although $W$ is not bounded as in inequality \eqref{eq:metric matrix}, it still satisfy conditions for the existence of geodesic curves. To see this claim, the Hopf-Rinow theorem states if a geodesic between two points can be extended on the whole domain, then there exist a geodesics between any two points. Note that there exists matrices $F$ and $G$ such that the maximum eigenvalue of $W$ satisfy the inequality $\lambda_{\max}(W(q))\leq|Fq+G|^2$, for every $q\in\mathbb{R}^9$. From the definition  $M=W^{-1}$, the minimum eigenvalue of $M$ satisfies the inequality $\lambda_{\min}(M(q))\geq |Fq+G|^{-2}$, for every $q\in\mathbb{R}^9$. Take any geodesic $c$, it is the solution to the equation $(c')^\top M(c)c'=o$, where $o$ is a constant value. From the inequality of the eigenvalue, the bound $|c'|\leq o|Fc+G|$ holds and $c$ ``grows'' exponentially, in the worst-case scenario. Thus, $c$ exists for all values of its argument and Hopf-Rinow theorem ensures the existence of a geodesic between any two points or $\mathbb{R}^9$.

%%%%%%%%%%%%%%%%%%%%%%%%%%%%%%%%%%%%%%%%%%%%%%%%%%%%%%%%%%%%%%%%%%%%%%%%%%%%%%%%
\section{Conclusion \& Perspectives}\label{sec:conclusion}

In this work, the authors have shown how to employ contraction metrics to design a scalable distributed controller. The contributions of this paper are two-fold. A constructive method to design a distributed feedback law, and a distributed convex-optimization problem to compute the contraction metric with diagonal (separable) structure. Note also this approach can be applied to system described by less regular functions.

In a future work, the authors aim to extend this approach to functions contraction metrics with different structures such as maximization. The authors also aim to find conditions for the existence of the matrix $M$ satisfying diagonal and other structure constraints. In addition to that, the authors also intend to improve the controller design framework to take into account robustness and saturation constraints.

\paragraph*{Acknowledgments}

The authors gratefully acknowledge the contribution of Australian Research Council and the comments of anonymous reviewers.

%%%%%%%%%%%%%%%%%%%%%%%%%%%%%%%%%%%%%%%%%%%%%%%%%%%%%%%%%%%%%%%%%%%%%%%%%%%%%%%%

\printbibliography

\end{document}